\documentclass[runningheads]{llncs}
\usepackage{subfigure}
\usepackage{tikz}
\usepackage[utf8]{inputenc}          
\usepackage[T1]{fontenc}             
\usepackage[USenglish]{babel}        
\usepackage[intlimits]{amsmath}      
\usepackage{amsfonts}                

\usepackage{amssymb}                 
\usepackage{amsmath}

\usepackage{mathtools}
\usepackage{setspace}
\usepackage{graphicx}
\usepackage{hyperref}

\usepackage{comment}

\usepackage{mathrsfs}

\usepackage{algorithm}
\usepackage[noend]{algpseudocode}

\newtheorem{redrule}{Reduction Rule}

\begin{document}
\title{A Linear Kernel for Planar Vector Domination}
\author{Mahabba El Sahili
\and Faisal N. Abu-Khzam}
\authorrunning{El Sahili and Abu-Khzam}
\institute{
Department of Computer Science and Mathematics\\
Lebanese American University, 
Beirut, Lebanon.\\
\email{\{faisal.abukhzam,mahabba.elsahili\}@lau.edu.lb}
}
\maketitle 

\begin{abstract}
Given a graph $G$, an integer $k\geq 0$, and a non-negative integral function $f:V(G) \rightarrow \mathcal{N}$, the {\sc Vector Domination} problem asks whether a set $S$ of vertices, of cardinality $k$ or less, exists in $G$ so that every vertex $v \in V(G)\setminus S$ has at least $f(v)$ neighbors in $S$.
The problem generalizes several domination problems and it has also been shown to generalize {\sc Bounded-Degree Vertex Deletion} (BDVD). In this paper, the parameterized version of Vector Domination is studied when the input graph is planar. A linear problem kernel is presented. A direct consequence is a kernel bound for BDVD that is linear in the parameter $k$ only. Previously known bounds are functions of both the target degree and the input parameter.

\end{abstract}

\section{Introduction}
The {\sc Vector Dominating Set} problem assumes a graph $G$ is given along with an integer $k\geq 0$ and a demand function $f$ from the set $V$ of vertices of $G$ to the set $\{0,\ldots |V|-1\}$, and asks whether a set $D$ of at most $k$ vertices of $G$ exists such that each vertex $v\in V(G)\setminus D$ has at least $f(v)$ neighbors in $D$. The problem can be viewed as a generalization of many domination problems that are extensively studied in the literature, including the classic {\sc Dominating Set} problem, the $r$-{\sc Dominating Set} problem \cite{LAN20131513,chellali2010p,chen1998upper,favaron2008k,hansberg2009upper} and 
the $\alpha$-{\sc Dominating set} problem \cite{DUNBAR200011,DBLP:journals/dmgt/DahmeRV04,DAHME20083187,article333} which is also known to be a generalization of {\sc Positive Influence Dominating Set} (PIDS). The latter problem is known for applications in social networks \cite{7411175,PIDS,10.1007/s10898-009-9511-2,AL18}. PIDS is also called a monopoly in the literature \cite{monopoly}. In addition to domination problems, {\sc Vector Domination} is also a generalization of the {\sc Bounded Degree Vertex Deletion} problem (BDVD)
\cite{ganian2021structural,BETZLER201253,FELLOWS20111141},
along with its dual, the $s$-{\sc Plex Detection} problem. The latter also finds application in social network analysis \cite{article,article2}. 

From a computational complexity standpoint, {\sc Vector Dominating Set} is $NP$-hard being a generalization of the dominating set problem. 
It is well known by now that the study of a problem's parameterized complexity \cite{dblp1894779,DowneyF13} provides an alternative view of its complexity and can shed light on classes/types of input where the problem can be tractable. 
Under the parameterized complexity lens, a problem is fixed parameter tractable (FPT) if there exists an algorithm that solves it in $O(f(k)n^c)$ time where $n$ is the size of the input, $k$ is an input parameter, $f$ is a function of $k$ and $c$ is a constant. Alternatively, a problem is FPT if it admits a kernel i.e given an instance $(G,k)$ of the problem, an equivalent instance $(G',k')$ can be constructed in polynomial time where $|G'| \leq g(k)$ and $k' \leq k$. $G'$ is referred to as a problem kernel. Of particular interest is the case where $g$ is a (low degree) polynomial function of $k$.
The {\sc Dominating Set} problem is $W[2]$-hard on general graphs \cite{flum_2002}, thus there is little or no hope in achieving fixed parameter tractability for it or for the more general {\sc Vector Dominating} set problem in general graphs. 

On the other hand, {\sc Vector Dominating Set} is known to be fixed-parameter tractable when the input is restricted to planar graphs \cite{ISHII201680}.
By exploiting structural properties of planar graphs, especially degree structure, Alber et al. \cite{alber2002polynomial} provided a linear kernel for the planar {\sc Dominating Set} problem. This early work was followed by a series of papers on linear kernels for different domination  problems  \cite{alber2002polynomial,GARNERO2017536,garnero2018linear,LOKSHTANOV20112536}. Surprisingly, no linear kernel is known yet for the (more general) {\sc Vector Dominating Set} problem in planar graphs.  Addressing this problem is the main objective of this paper. A linear kernel is presented  based on several structural properties of planar graphs, including the methods used in \cite{alber2002polynomial}. The bound achieved is $101k$, and it results in  linear kernel bounds for several special problems including the {\sc Planar Bounded-Degree Vertex Deletion} (PBDVD) for which the obtained kernel is linear in the number of deleted vertices ($k$) only. Previously known kernel bounds (in general) are functions of $k$ and the target degree bound \cite{?}.



\section{Preliminaries}

Common graph-theoretic terminology is used in this paper, with a main focus on planar graphs and their properties.
Although we assume the input to the considered problem is a planar graph that is given with a fixed embedding (i.e., plane graph), we note that deciding whether a given graph is planar, and constructing a corresponding planar embedding, can be done in linear time \cite{10.1145/321850.321852}. It is well known that the number of edges in a planar graph $G$ with $|G|\geq 3$ is at most $3|G|-6$, where $|G|$ is the order of $G$ (i.e., number of vertices). Moreover, if $G$ is bipartite the number of edges is at most $2|G|-4$. The following problem is considered:

\vspace{10pt}
\noindent
{\sc Planar Vector Dominating Set} (PVDS) 

\noindent
{\bf Given:} A planar graph $G=(V,E)$ of order $n$, an $n$-dimensional non-negative {\em demand} vector $d$, 
and an integer $k$.

\noindent
{\bf Question:} Is there a set $S\subset V$ of cardinality $k$ such that every vertex $v \in G-S$ has at least d(v) neighbors in $S$?

\vspace{10pt}

Let $(G,d,k)$ be an instance of PVDS. A vertex $v\in V(G)$ is said to be a $j$-vertex if $d(v) = j$ where $j\in \mathbb{N}$. A $j$-vertex will be considered {\em demanding}, or of {\em high-demand} if $j > 0$, otherwise it will be a {\em low-demand} vertex.
We denote by $G-v$ the subgraph of $G$ induced by $V(G)\setminus \{v\}$. The set of neighbors of $v$ is denoted by $N(v)$ while $N_l(v)$ and $N_h(v)$ are sets of low-demand and high-demand neighbors of $v$, respectively. Furthermore, we denote by $N[v]$ the set $N(v)\cup \{v\}$. The same applies to $N_l[v]$ and $N_h[v]$.

For a set of vertices $A\subset V(G)$ we define  $N(A)=\cup_{v\in A} N(v)$, and we denote by $N_l(A)$ and $N_h(A)$ the sets of low-degree and high-degree elements of $N(A)$, respectively. $N[A], N_l[A]$ and $N_h[A]$ are defined analogously.

A solution is said to {\em observe} an edge $e=uv$ where $u,v\in V(G)$ if one of the endpoints of $e$, namely $u$, is deleted to decrease d(v). We say that a set $A$ dominates another set $B$ if for all $v \in B$, $|N(v) \cap A| \geq d(v)$.

In some cases, we know certain vertices are not to be selected into a PVDS solution. This gives rise to a yet another more general problem. In particular, we shall consider the following {\em annotated} PVDS problem:

\vspace{10pt}

\noindent
{\sc Annotated Planar Vector Dominating Set} (APVDS)

\noindent
{\bf Given:} A planar graph $G=(V,E)$ of order $n$, an $n$-dimensional non-negative {\em demand} vector $d$, and an integer $k$, and a set $P \subset V$.

\noindent
{\bf Question:} Is there a set $S\subset V-P$ of cardinality at most $k$ such that every vertex $v \in G-S$ has at least d(v) neighbors in $S$?

\vspace{10pt}

{\em Redundant} solution elements play a role in our kernelization method. Let $v$ be a vertex such that: if there is a solution $S$ of APVDS where $v\in S$, then there is anther solution $S'$ where $v \notin S'$. We will further annotate $G$ by assigning a special blue color to such vertices, which are assumed to be redundant (or replaceable) in this case.
Note that if $v$ is redundant, then the two APVDS instances $(G,k,d,P)$ 
and $(G,k,d,P\cup \{v\})$ are equivalent.

\section{A Linear Kernel}

The following reduction rules for the PVDS are assumed to be applied successively and exhaustively. The first five rules are applicable to the  general {\sc Vector Dominating Set} problem.


\begin{redrule}
\label{1}
If $u$ and $v$ are adjacent 0-vertices then delete the edge $uv$.
\end{redrule}

\begin{redrule}
\label{2}
Delete all isolated 0-vertices.
\end{redrule}

\begin{redrule}
\label{3}
If for some vertex $v$, $d(v)>k$ or $d(v)>|N(v)|$, delete $v$, then decrease $d(u)$ by one for all $u \in N(v)$ and decrease $k$ by one.
\end{redrule}

\begin{redrule}
\label{4}
Let $v$ be a 0-vertex. If there exists a vertex $a \in V(G)$ such that $N(v) \subset N[a]$, then for any 1-vertex $b$ in $N(v)$, delete the edge $vb$. Delete the edge $va$ if it exists (see Figure \ref{fig1}).
\end{redrule}

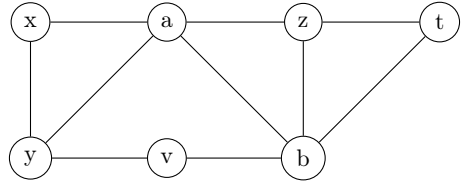
\begin{figure}[bth!]
\begin{center}
	\begin{tikzpicture}[transform shape,scale=0.9]
			\tikzset{every node/.style={circle,minimum size=0.4cm}}
			\node[draw] (A) at (-2,1) {x};
			\node[draw] (C) at (-2,-1) {y};
			\node[draw] (B) at (0,1) {a};
			\node[draw] (D) at (0,-1) {v};
			\node[draw] (E) at (2,1) {z};
            \node[draw] (G) at (4,1) {t};
			\node[draw] (F) at (2,-1) {b};
			
			\path (A) edge[-] (C);
			\path (B) edge[-] (C);
			\path (B) edge[-] (A);
			\path (B) edge[-] (E);
			\path (B) edge[-] (F);
			\path (D) edge[-] (C);
			\path (D) edge[-] (F);
			\path (F) edge[-] (E);
            \path (F) edge[-] (G);
            \path (G) edge[-] (E);
        \end{tikzpicture}
\end{center}
\caption{Reduction Rule 4: delete $vb$.} 
\label{fig1}
\end{figure}

\noindent
{\bf Soundness.} Any potential solution must contain a neighbor of vertex $b$. If this neighbor is $v$, then it can be very well replaced by $a$. Therefore, if a solution exists, we know for sure there must be a solution that does not observe the edges $bv$ and $av$.

\begin{redrule}
Let $v$ be a 1-vertex. If there exists a vertex $a \in N(v)$ such that every vertex $u\in N[v]-a$ satisfies: 
$d(u) \leq 1$, and $N_h[u]\subset N[a]$,
then delete $a$, decrease the demand of each neighbor of $a$ by one, and decrease k by one. 
\end{redrule}

\noindent
{\bf Soundness.} 
Any solution $S$ must contain a vertex $b$ of $N[v]$. If $a\notin S$, replace $b$ by $a$ since $N_{h}[b]\subset N[a]$.

\begin{definition}
For $s_{1}, s_{2} \in V(G)$, $s_{1} \neq s_{2}$:

\begin{itemize}

\item An $s_{1} s_{2}-$path of length two is called a type 1 path.

\item An $s_{1} s_{2}-$path of length 4 $s_{1} v c v' s_{2}$ is called a type 2 path if $d(c)=0$, $d(v)=1$, $v \notin N(s_{2})$ and $v' \notin N(s_{1})$. 

\item An $s_{1} s_{2}-$path of length 3 $s_{1} v  v' s_{2}$ is called a type 3 path if $d(v)\leq 1$. A 1-vertex outside this path adjacent to both $s_{1}$ and $v$ is called a "leech" for this path.
\end{itemize}
\end{definition}

\begin{definition}
    
Suppose the graph contains a solution set $S$. A real path $p$ is a path with endpoints in $S$ satisfying one of the following conditions:

\begin{itemize}

\item $p$ is a type 1 path;

\item $p$ is a type 2 path such that $p=s_{1} v c v' s_{2}$ where $v$ and $v'$ are not both used by a type 1 real path;

\item  $p$ is a type 3 path such that $p=s_{1} v v' s_{2}$ where $v$ and $v'$ are not both used by a type 1 real path.
\end{itemize}

\end{definition}

From this point on, we denote by $A$ the set of 0-vertices such that $N(A)\subset \{u\in V(G): |N(u)\cap S|\geq 2\}$.
Set $L$ to be the set of all leeches for real paths.


\begin{lemma}
\label{edge-num}
If V(G) contains a solution set $S$, then after applying the reduction rules above any vertex $v \in G-(S\cup A \cup L)$  belongs to some real path.
\end{lemma}

\begin{proof}
Suppose V(G) contains a deletion set $S$. Let $v \in G \setminus (S\cup A \cup L)$. If $d(v) \geq 2$, $v$ must be adjacent to two vertices in $S$ so $v$ is on a real path of type 1.
If $d(v)=1$, it must be adjacent to some $s_{1} \in S$. We may assume that $N(v)\cap S=\{s_{1}\}$ otherwise it lies on a type 1 real path. If it is connected to a 0-vertex $c$, $c \notin S$, due to reduction rule 4 we have $N(c) \not\subset N[s_{1}]$. Let $v'\in N(c)\setminus N[s_{1}]$, 
Due to reduction rule 1 $v'\in N[s_{2}]$ for some vertex $s_{2} \in S$, $s_{2} \neq s_{1}$. We may easily verify that $v$ is  on some real path. Now let $u\in N(v)\setminus S$. If $|N(u)\cap S| \geq 2$, $v$ lies on a type 3 real path, so we may assume that $d(u)=1$ for all $ u\in N(v)-s_{1}$. Due to reduction rule 5, there exists a vertex $ u'\in N(v)-s_{1}$ where $u'\in N(s_{2})$ for some $s_{2}\in S, s_{2} \neq s_{1}$. Thus $v$ is a leech for the real path $s_{1} u u' s_{2}$, a contradiction.

\noindent
Now suppose that $d(v)=0$. If $|N(v)\cap S| \geq 2$, $v$ lies on a type 1 real path, and as above  if $N(v)\cap S =\{s_{1}\}$ then there exists a real path $s_{1} v u s_{2}$. So we may suppose that $N(v)\cap S =\emptyset$. Since $v \notin A$, there exists $u\in N(v)$ where $N(u)\cap S=\{s_{1}\}$, As above using reduction rule 4 we may verify that there exists a real path $s_{1} u v u' s_{2}$.
\end{proof}

Following the approach of Alber et al (2004), we define the concept of a region.
\begin{definition}
Let $G(V,E)$ be a plane graph with solution set $S$. A region $R(v,w)$ between two vertices $v,w\in S$ is a closed subset of the plane with the following properties:

\begin{itemize}
    
\item The boundary of $R(v,w)$ is formed by two $vw$-real paths $o_{1}$ and $o_{2}$;

\item All vertices strictly inside the region $R(v,w)$ can be dominated by $\{v,w\}$;

\item There is a line in the plane separating the region $R(v,w)$ into two closed areas $A_{1}$ and $A_{2}$ where $o_{1}$ lies on the boundary of $A_{1}$ and $o_{2}$ lies on the boundary of $A_{2}$.

\end{itemize}

\end{definition}

From this point on, we refer to the above paths $o_{1}$ and $o_{2}$ as {\em outer paths} between $v$ and $w$. For a region $R = R(v, w)$, let $V(R)$ denote the vertices belonging to $R$, i.e.,
$V(R) = \{u \in V : u$ sits inside or on the boundary of $R\}$. In what follows, the boundary of a region $R$ will be denoted by $\partial R$.


\begin{definition} \cite{alber2002polynomial}
Let $G = (V, E)$ be a plane graph and $S \subset V$. An S-region decomposition of $G$ is a set $\mathscr{R}$ of regions between pairs of vertices in $S$ such
that:
\begin{enumerate}
    \item  for $R(v, w) \in$ $\mathscr{R}$ no vertex from $S$ (except for v, w) lies in V (R(v, w))
    \item for two regions $R_{1}, R_{2} \in $$\mathscr{R}$, $(R_{1} \cap R_{2}) \subset (\partial R_{1} \cup \partial R_{2})$.
\end{enumerate}
 An S-region decomposition $\mathscr{R}$ is called maximal if it is maximal for inclusion.\\
One defines the graph $G_{\mathscr{R}} = (V_{\mathscr{R}}, E_{\mathscr{R}})$ with possible multiple edges of an S-region decomposition $\mathscr{R}$ of $G$ where $V_{\mathscr{R}}= S$  and
$E_{\mathscr{R}}= \{\{v, w\}:$  there is a region $R(v, w) \in \mathscr{R}$  between $v, w \in S\}$.
\end{definition}

We shall show that after applying the reduction rules above to an instance $(G,k)$ of PVDS, for which a solution set $S$ exists, there exists a maximal $S$-region decomposition $\mathscr{R}$ such that:
\begin{enumerate}
    \item The number of regions of $\mathscr{R}$ is $O(k)$.
    \item Each region $R \in $$\mathscr{R}$ contains $O(1)$ vertices.
     \item The number of vertices not in $V(\mathscr{R})$ is $O(k)$.
\end{enumerate}
 These facts together mean that the size of the graph after applying the reduction rules is $O(k)$, and a linear kernel is obtained.

\begin{definition} \cite{alber2002polynomial}
A planar graph $G = (V, E)$ with multiple edges is thin if there exists a planar embedding such that if there are two edges $e_{1}$, $e_{2}$ between a pair of distinct vertices $v, w \in V$, then there must be two further vertices $u_{1}, u_{2} \in V$ which sit inside
the two disjoint areas of the plane that are enclosed by $e_{1}$, $e_{2}$.
\end{definition}
\begin{lemma}\cite{alber2002polynomial}
For a thin planar graph G = (V, E) we have $|E|\leq 3|V |$-$6$.
\end{lemma}

The following proposition can be easily deduced from \cite{alber2002polynomial}.
\begin{proposition}
For a plane graph G with vector dominating set S, there exists a maximal S-region decomposition $\mathscr{R}$ such that $G_{\mathscr{R}}$ is thin.
\end{proposition}

Consequently, the number of regions $R(v,w)$ will be bounded by $3k-6$.

\begin{lemma}
Suppose $G$ admits a solution $S$ with a maximal S-region decomposition $\mathscr{R}$, then any vertex $v \in G-(S\cup A \cup L)$ is in $V(\mathscr{R})$.
\end{lemma}

\begin{proof}
Take $v\notin L \cup A \cup S$, then by Lemma 1 $v$ must belong to some real path $p$. We will say that an edge crosses a region $R$ if the edge lies (except possibly for its endpoints) strictly inside $R$. Similarly, we say that a path crosses a region $R$ if at least one edge of the path crosses $R$. Suppose that $v\notin V(\mathscr{R})$, then $p$ must cross some region $R=R(s_{1},s_{2})$ where $s_{1},s_{2} \in S$. Let $o_{1}$ and $o_{2}$ be the two outer paths of $R$, then $p$ must intersect $o_{1}$ or $o_{2}$. Without loss of generality, we will assume $p$ intersects $o_{1}$. If $p$ is a type 1 path the Lemma follows trivially. Thus we will consider the following cases:
\begin{itemize}
\item Case 1: $p$ is a type 2 path $s v c v' s'$.\\
In this case either $c\in o_{1}$ or $v' \in o_{1}$.
If $c\in o_{1}$ and $o_{1}$ is a type 1 path a real path disjoint from $\mathscr{R}$ arises, a contradiction.

If $o_{1}$ is a type 2 path, $o_{1}=s_{1} u c u' s_{2}$, then one of the regions $s v c u s_{1}$ or $s v c u' s_{2}$ contains $v$ contradicting the maximality of $\mathscr{R}$. Indeed we may suppose that $s \neq s_{2}$, if $u'$ and $s$ are not adjacent, then $s v c u' s_{2}$ is a real path disjoint from $\mathscr{R}$. Otherwise $s \neq s_{1}$ then $s v c u s_{1}$ has the desired property as $u$ and $s$ are not adjacent due to the fact that $o_{1}$ is real. Now if $o_{1}$ is a type 3 path $s_{1} c u s_{2}$ then if $s \ne s_{1}$ the path $s v c s_{1}$ is a region that contains $v$, else $s=s_{1}$ then $v$ is a leech for $o_{1}$ so $v\in L$ contradicting our assumption.
Now if $c\notin o_{1} $ and $ v'\in o_{1}$. In this case the edge $v's'$ crosses $R$ so $s' \in V(R)$ thus $s'$ is an end of $o_{1}$. Then $o_{1}$ is not of type 1 or 2, since in these cases $v's' \in o_{1}$. The remaining case where $o_{1}$ is a type 3 path may be treated exactly as above.

\item Case 2: p is a type 2 path $s u v u' s'$ or a type 3 path $s v u' s'$.

In this case suppose (namely) that the edge $u's'$ crosses $R$, then, as in the previous argument, a region containing $u$ and $v$ can be found that does not cross any other region in $\mathscr{R}$, contradicting the maximality of $\mathscr{R}$, or the same reasoning as above can be repeated.
\end{itemize}
\end{proof}

\begin{definition}
let $a_{1}, a_{2} \in V(G)$, Call a closed subset of the plane $C(a_{1},a_{2})$ a candidate region between $a_{1}$ and $a_{2}$ if the boundary of $C(a_{1},a_{2})$ is formed of two $a_{1}a_{2}-$paths of type 1,2 or 3, every vertex in the interior of $C(a_{1},a_{2})$ is dominated by $\{a_{1},a_{2}\}$, and $C(a_{1},a_{2})$ is maximal for these properties (for inclusion).
\end{definition}

Let $a_{1}, a_{2} \in V(G)$ and take $C(a_{1},a_{2})$ to be a candidate region between $a_{1}$ and $a_{2}$ with boundary $O(a_{1},a_{2})=O_{1} \cup O_{2}$. Consider the following sets:
\begin{itemize}
    \item $Y(a_{1},a_{2})=\{v: v\in O(a_{1},a_{2})$ and $d(v) \geq 2\}$
    \item $B(a_{1},a_{2})$ is the set of vertices in the interior of $C(a_{1},a_{2})$ that are adjacent to some vertex in $O(a_{1},a_{2})$.
    \item $I(a_{1},a_{2})$ is the set of vertices in the interior of $C(a_{1},a_{2})$ that are not in $B(a_{1},a_{2})$.
    \item $I'(a_{1},a_{2})=$ $\{v$: $\{v\}$ dominates $I(a_{1},a_{2}), v\notin P\}$
    \item $O'(a_{1},a_{2})=\{v: |N(v) \cap Y(a_{1},a_{2})| \geq 2\}$
\end{itemize}

In the following rules let $a_{1}, a_{2} \in V(G)$, $C(a_{1},a_{2})$  be a candidate region between $a_{1}$ and $a_{2}$. Let $C, O, B, I,I'$ and $O'$ denote $C(a_{1},a_{2})$, $O(a_{1},a_{2})$, $B(a_{1},a_{2}),$ $I(a_{1},a_{2}),$ $I'(a_{1},a_{2})$  and $O'(a_{1},a_{2})$ respectively. Reduction rules $6,7$ and $8$ apply only on planar graphs. 

Note that if $I \neq \emptyset$ and a vertex $w \notin I'$ is in the solution set $S$, then there must be another vertex $w'\in I\cup B \cup \{a_{1},a_{2}\}$ in $S$. The following reduction rule follows:

\begin{redrule}
Suppose $I\neq \emptyset$ and let $u\in C$. If $u\notin I'$ and $N(u)\cap Y=\emptyset$, then color $u$ blue.
\end{redrule}

\noindent
{\bf Soundness.}
Suppose $u\in S$, since $u\notin I'$, $\exists u'\in I\cup B \cup \{a_{1},a_{2}\}$ in $S$.
We have $Y\cap N(u) =\emptyset $, then if $u'=a_{1}$ (namely) replace $u$ by $a_{2}$ (all $N[u]$ is dominated by $\{a_{1},a_{2}\}$). So suppose $u'\in B\cup I$, then replace $u$ and $u'$ by $a_{1}$ and $a_{2}$, since $B\cup I$ would be dominated and $|N(Y)\cap \{u,u'\}| \leq 1$ while $ \forall y\in Y, |N(y)\cap \{a1,a2\}| \geq 1$.

\begin{redrule}
Suppose  $I\neq \emptyset$ and $O'\neq \emptyset$. If a vertex $w$ is inside $C$, then color it blue unless one of the following holds:

\begin{itemize}
    \item $w\in O'\cup I'$.
    \item $\exists w_{2}\in N(Y) $ where $\{w,w_{2}\}$ dominates $I$.
    \item $\exists w_{2}\in O'$ where $\{w,w_{2}\}$ dominates $I\cap N(a_{1})$ or $I\cap N(a_{2})$.
\end{itemize}
\end{redrule}

\noindent
{\bf Soundness.}
Suppose $w\notin O'\cup I'$ and $w\in S$. We will show that if $w$ is a potential solution vertex then one of the above cases must hold. $w$ is a potential solution vertex so by reduction rule 6 it must have a neighbor in $Y$. Let $y\in N(w)\cap Y$ with $y\in N(a_{1})$ (namely) and $y'$ be the other vertex of $Y$. Let $W$ be the vertices in $S\cap C$ except for vertices in $O-Y$. We have $|W|\leq 3$ since any four vertices in $W$ can be replaced by $a_{1},a_{2}, y$ and $y'$. If the vertices of $W$ are all strictly inside $C$ with $W\cap O'=\emptyset$ and $|W| \geq 3$, then we must have either $|N(y)\cap W|\leq 1$ or $|N(y')\cap W|\leq 1$. Suppose namely it is the first case then the vertices of $W$ can be replaced by $a_{1}, a_{2}, $ and $y'$. Now we may assume that $W$ contains vertices of $O\cup O'\cup \{a_{1}, a_{2}\}$. If $W\cap O\neq \emptyset$, $w$ and at most 1 other vertex in $W$ must dominate $I$ ($I\cap N(O)=\emptyset$) thus $w$ satisfies one of the cases of the reduction rule. Now if $W\cap O= \emptyset$ then if we take $w', w''$ in $W$ different from $w$, it can be easily verified that any combination of choices of $w',w''$ either leads to replacing a vertex of $W$ or to showing that $w$ follows one of the above cases. An example would be if $w',w'' \in O'$ then if $Y\subset N(a_{1})$ suppose namely $w'$ is inside the region $a_{1}yw''y'$ then replace $w'$ by $a_{1}$, which takes us to another case. So we must have $y\in N(a_{1})$ and $y'\in N(a_{2})$ in this case since the vertices of $O'$ separate $C$ into two regions, if $w$ is in the region of $a_{1}$ replace it by $a_{1}$, otherwise replace $w$ and the vertex of $O'$ in the region of $a_{2}$ by $y$ and $a_{2}$. Other cases can be treated similarly.

\begin{redrule}
Suppose $I\neq \emptyset$ and $O'=\emptyset$. Color any vertex $u$ in $C$ blue except if one of the following is true:
\begin{itemize}
    \item $u\in I'$
     \item $\exists u'\in N(y)$ where $y\in N(u)\cap Y$ and $\{u,u'\}$ dominates $I$.
     \item $Y\subset N(a_{1})$ and $\exists u'\in N(Y)$ where $\{u,u'\}$ dominates $I-N(a_{1})$
\end{itemize}
\end{redrule}

\noindent
{\bf Soundness.}
Suppose $O’=\emptyset$, let u be a vertex that is not blue, $u\notin I’$. Then by reduction rule 6 rule $N(u)\cap Y =\{y\}$ for some $y\in Y$. There must be some $u’\in I\cup B\cup \{a1,a2\}$ which is in $S$. Let $W$ be as in the above reduction rule and we have that no four vertices of $W$ can be in $S$. If $Y \subset N(a_{1})$ then if $a_{1}\in W$, $u$ and $u'$ dominate $I-N(a1)$ and $u'\notin P$ so $u'\in N(Y)$ by reduction rule 6, so $u$ follows one of the above cases. Else if $|W|=3$ replace the three vertices of $W$ by $a_{1}$, $a_{2}$ and a vertex of $Y$ having 2 or more neighbors in $W$ so $|W|= 2$ and we are also done since $u'\in N(y)$ otherwise replace $u$ and $u'$ by $a_{1}$ and $a_{2}$. Now if $Y\not\subset N(a_{1})$ suppose $|W|=3$ then as before we can always replace the vertices of $W$ including u, so $|W|=2$ and $u'\in N(y)$ otherwise $u$ and $u'$ can be replaced.\\

Reduction rules $9-13$ concern blue vertices. They are applicable on general graphs and should be applied successively and exhaustively.

\begin{redrule}
Let $v\in V(G)$ where $d(v)\leq 1$. If $\exists w$ where $N_{h}[v] \subset N(w)$ and there exists at most one vertex $z \in N(v)-w$ with $d(z) \geq 2$, then color $v$ blue.
\end{redrule}

\noindent
{\bf Soundness.}
Suppose $v\in S$. If $w \notin S$, replace $v$ by $w$, otherwise replace $v$ by $z$.

\begin{redrule}
Delete all the edges between blue vertices, and delete 0-vertices that are blue.
\end{redrule}

\begin{redrule}
Let $v$ be a blue vertex, if $N(v)=\{u,w\}$ for some $u,w \in V(G)$ and the edge $uw$ exists, delete it and decrease each of $d(u)$ and $d(w)$ by one.
\end{redrule}

\noindent
{\bf Soundness.}
$v$ is a blue vertex so it is a high demand vertex due to reduction rule 10, thus one of its neighbors must be in $S$. Suppose (namely) $u$ is in $S$, then in $G-S$ the edge $uw$ is deleted and $w$ must have $d(w)-1$ other neighbors in $S$, so we can replace $d(w)$ by $d(w)-1$, and $u$ would not need any other neighbor in $S$ so decreasing $d(u)$ has no effect on the solution.

\begin{redrule}
Let $v$ be a blue vertex. If for some 1-vertex $u \neq v$ $N(v) \subset N[u]$, then replace $d(u)$ by 0.
\end{redrule}

\noindent
{\bf Soundness.}
$v$ is blue which means it is a high demand vertex by reduction rule 10. Thus $N(v) \cap S \neq \emptyset$ so $N(u) \cap S \neq \emptyset$ regardless of $d(u)$, so it will make no difference to decrease $d(u)$ to zero.

\begin{redrule}
Let $u$ be a 0-vertex where $|N(u)|=2$. If there exists a 0-vertex $v$ where $N(u)=N(v)$, then delete $u$.
\end{redrule}

\noindent
{\bf Soundness.}
Suppose $u\in S$, then if $v \notin S$ replace $u$ by $v$, otherwise replace $u$ and $v$ by $N(u)$. Thus $u\notin S$ and since $d(u)=0$ we can delete it.

\begin{lemma}
After applying the reduction rules above to a graph any candidate region $C$  contains at most $15$ vertices.
\end{lemma}

The above lemma follows by the planarity of the graph and the above reduction rules. Below is the worst case example for the region $C$ when $O' \neq \emptyset$. In this case $w$ dominates $I$ with a vertex of $O'$ and the same applies to $u$. (see Figure \ref{fig2})

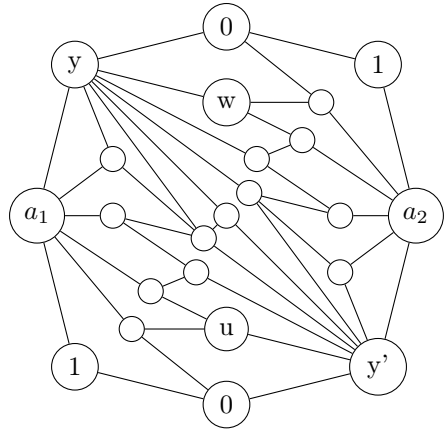
\begin{figure}[bth]
\begin{center}
	\begin{tikzpicture}[transform shape,scale=1]
			\tikzset{every node/.style={circle,minimum size=0.05cm}} 
			\node[draw] (A) at (-2.5,0){$a_{1}$};
            \node[draw] (B) at (-2,2){y };
            \node[draw] (C) at (0,2.5){0};
            \node[draw] (D) at (2,2){1};
            \node[draw] (E) at (2.5,0){$a_{2}$};
            \node[draw] (F) at (2,-2){y'};
            \node[draw] (G) at (0,-2.5){0};
            \node[draw] (H) at (-2,-2){1};
            \node[draw] (v'') at (0,0) {};
            \node[draw] (v) at (0.3,0.3) {};
            \node[draw] (v') at (-0.3,-0.3) {};
            \node[draw] (b) at (1.5,0) {};
            \node[draw] (c) at (1.5,-0.75) {};
            \node[draw] (w) at (0,1.5) {w};
            \node[draw] (d) at (0.4,0.75) {};
            \node[draw] (e) at (1,1) {};
            \node[draw] (f) at (1.25,1.5) {};
            \node[draw] (b') at (-1.5,0) {};
            \node[draw] (c') at (-1.5,0.75) {};
            \node[draw] (w') at (0,-1.5) {u};
            \node[draw] (d') at (-0.4,-0.75) {};
            \node[draw] (e') at (-1,-1) {};
            \node[draw] (f') at (-1.25,-1.5) {};
            \path (A) edge[-] (B);
            \path (B) edge[-] (C);
            \path (C) edge[-] (D);
            \path (D) edge[-] (E);
            \path (E) edge[-] (F);
            \path (F) edge[-] (G);
            \path (G) edge[-] (H);
            \path (H) edge[-] (A);
			\path (B) edge[-] (v);
            \path (F) edge[-] (v);
            \path (B) edge[-] (v');
            \path (F) edge[-] (v');
            \path (b) edge[-] (v);
            \path (b) edge[-] (E);
            \path (c) edge[-] (v);
            \path (c) edge[-] (E);
            \path (c) edge[-] (F);
            \path (B) edge[-] (v'');
            \path (F) edge[-] (v'');
            \path (v'') edge[-] (v');
            \path (B) edge[-] (d);
            \path (B) edge[-] (w);
            \path (d) edge[-] (b);
            \path (e) edge[-] (w);
            \path (e) edge[-] (d);
            \path (e) edge[-] (E);
            \path (f) edge[-] (w);
            \path (f) edge[-] (C);
            \path (f) edge[-] (E);
            \path (b') edge[-] (v');
            \path (b') edge[-] (A);
            \path (c') edge[-] (v');
            \path (c') edge[-] (A);
            \path (c') edge[-] (B);
            \path (F) edge[-] (d');
            \path (F) edge[-] (w');
            \path (d') edge[-] (b');
            \path (e') edge[-] (w');
            \path (e') edge[-] (d');
            \path (e') edge[-] (A);
            \path (f') edge[-] (w');
            \path (f') edge[-] (G);
            \path (f') edge[-] (A);
        \end{tikzpicture}
\end{center}
\caption{Worst case when $O' \neq \emptyset$}.
\label{fig2}
\end{figure}


\begin{theorem}
Let $G$ be a planar graph, then computing a vector dominating set for $G$ of order at most $k$ can be reduced to computing a vector dominating set of order at most $k$ to a planar graph $G'$ where $|G'| \leq 101k$.
\end{theorem}
\begin{proof}
We will show that after applying the above reduction rules we obtain $|G'| \leq 101k$. Suppose $G$ admits a solution $S$ with a maximal $S$-region decomposition $\mathscr{R}$. Let $x$, $y$ and $z$ be the numbers of type 1, 2 and 3 outer paths of regions in $\mathscr{R}$ respectively. Let $A'=A-V(\mathscr{R})$, $L'=L-V(\mathscr{R})$, $A_{1}=\{c\in A': |N(c)| \geq 3\}$ and $A_{2} = A'\setminus A_{1}$. Let $K=\{v: d(v)=1$ and $v$ belongs to some outer path of $\mathscr{R}\}\cup S$. 
The number of outer paths of $\mathscr{R}$ is at most $6k$ (since we have at most $3k$ regions), with two 1-vertices at most on each outer path, thus $|K| \leq 12k+k=13k$.

\noindent
{\bf Claim1:} Each vertex of $A_{1}\cup L$ has three or more neighbors in $K$. To see this, take $c\in A_{1}$ then $N(c)\cap L =\emptyset$ (a leech cannot be adjacent to a 0-vertex) so $\forall v\in N(c)$, $v \in \mathscr{R}$ thus $v$ is on some outer path by planarity and since $|N(c)| \geq 3$ Claim1 follows.\\
Now take $l\in L'$. $l$ is a leech for some type 3 path, so $\exists s,v \in N(l)$ with $s\in S$ and $d(v)= 1$. $v$ must be on some outer path since $v \in V(\mathscr{R})$ ($v \notin A \cup L)$ and $v$ cannot be in the interior of a region due to planarity. Now, by reduction rule 11, $l$ must have a neighbor $v'$ different from $s$ and $v$. We must have $d(v')=1$ since $l$ is a leech and $v' \notin L$ since vertices of $L$ are colored blue by reduction rule 9, so no two leeches can be adjacent. Thus similar to $v$, $v'$ is on some outer path. This proves our claim.

\noindent
By considering the planar bipartite graph formed of the vertices of $A_{1}\cup L$ and $K$ and the edges between them, the number of edges $e$ is bounded above by $2(|A_{1}|+|L|+|K|)-4$, thus we have $3(|A_{1}|+|L|)\leq e \leq 2(|A_{1}|+|L|+13k)-4$ which gives $|A_{1}|+|L| \leq 26k$. 

\noindent
{\bf Claim2.} Let $c\in A_{2}$, then any vertex $v\in N(c)$ is on an type 1 or 3 outer path of a region $R\in \mathscr{R}$. This is true since $v$ is on a type 1 path $p$ by definition of $A$. If $p$ is an outer path, then Claim2 follows, otherwise $p$ must cross some region $R \in \mathscr{R}$ and $v$ must be on the boundary of $R$ by planarity. But then this cannot occur if the outer path containing $v$ is of type 2, which proves Claim2.

\noindent
Consequently (due to Claim2), the number of vertices in $N(A_{2})$ is at most $x+z$. Now no two vertices of $A_{2}$ can have the same neighborhood by reduction rule 13, so by viewing each vertex of $A_{2}$ as an edge between vertices of $N(A_{2})$, we can say that $|A_{2}| \leq 3(x+z)-6$.\\Thus we have $|V(G')| \leq 15\cdot3k+(x+3y+2z)+(3x+3z)+26k=71k+4x+3y+5z$ but this expression is maximized if $z=6k$, which gives $|V(G')| \leq 101k$.
\end{proof}

The above bound applies to several problems that can be seen as special cases. In particular, the same kernel bound is obtained for {\sc Bounded Degree Vertex Deletion}.

\begin{corollary}
The {\sc Planar Bounded Degree Vertex Deletion} problem admits a kernel of  order at most $101k$ irrespective of the target degree bound.
\end{corollary}

The same applies for the $r$-Dominating Set problem, which asks, for a given planar graph $G=(V,E)$ and integers $r,k$, whether there is a set $D$ of at most $k$ vertices such that every element of $V\setminus D$ has at least $r$ neighbors in $D$. Obviously, this is a special case of {\sc Vector Domination} where the demand vector $d$ satisfies $d(v) = r, \forall v\in V$.

\begin{corollary}
The {\sc Planar $r$-Dominating Set } problem admits a kernel of  order at most $101k$.
\end{corollary}

Another well-known special case of {\sc Vector Domination} is the {\sc $\alpha$-Dominating Set} problem. Given a graph $G=(V,E)$ along with an integer $k$ and $\alpha\in (0,1$], the $\alpha$-domination problem asks for a set $D\subset V$ of cardinality at most $k$ such that every vertex $w \in V\setminus D$ has a at least $\alpha\times degree(w)$ neighbors in $D$.

\begin{corollary}
The {\sc $\alpha$-Dominating Set } problem admits a kernel of order at most $101k$.
\end{corollary}

The same kernel bound is also obtained for the {\sc Positive Influence Dominating Set} problem (PIDS), which (typically) corresponds to the case where $\alpha = 0.5$.

\end{document}